\documentclass[12pt]{article}
\usepackage{amsmath}
\usepackage{amsfonts}
\usepackage{amssymb}
\usepackage{latexsym}
\usepackage{amscd}
\usepackage{color}
\usepackage{enumitem}
\usepackage[T1]{fontenc}

\input amssym.def
\newcommand{\norma}[1]{\| #1 \|}
\newcommand{\conj}[2]{\left \{ {#1} \, : \, {#2} \right \}}

\newcommand{\cvf}{\overset{\omega}{\rightarrow}}

\newcommand {\N} {\mathbb{N}}

\newcommand{\proof}{\noindent \textit{Proof: }}
\newcommand{\qed}{\hfill \ensuremath{\Box}}

\newtheorem{df}{Definition}[section]

\newtheorem{prop}[df]{Proposition}
\newtheorem{teo}[df]{Theorem}
\newtheorem{cor}[df]{Corollary}

\newtheorem{ex}[df]{Example}
\newtheorem{exs}[df]{Examples}
\newtheorem{rema}[df]{Remark}

\begin{document}

\title{A  class of sets in a Banach space coarser than limited sets}
\author{Pablo Galindo\footnote{Partially supported by Spanish MINECO/FEDER PGC2018-094431-B-I00.} ~and V. C. C. Miranda \footnote{Supported by a CAPES PhD's scholarship.}}
\date{}

\maketitle
\begin{abstract}
{A wide new class of subsets of a Banach space $X$ named coarse $p$-limited sets ($ 1\leq p < \infty$) is introduced by considering weak* $p$-summable sequences in $X'$ instead of  weak* null sequences. We study its basic properties and compare it with the class of compact and weakly compact sets. Results concerning the relationship of coarse $p$-limited sets with operators are obtained.

}\end{abstract}

\section{Introduction}
 Recall that a subset $A$ of a Banach space $X$ is said to be \textsl{limited} if weak*-null sequences in $X'$ converge uniformly to $0$ on $A.$ Or equivalently, if for every linear operator $T:X\to c_0,$ $T(A)$ is a relatively compact set. Notions alike to limitedness have been considered in several contexts. For instance, in the Banach lattice context the so-called \textsl{almost limited} sets were introduced by Chen, Chen and Ji \cite{chen}, while, around the approximation properties, the $p-$\textsl{limited} ($1 \leq p < \infty$) sets were defined by Karn and Sinha \cite{karnsinha2} and then studied by Delgado and Pi\~neiro \cite{delgado}.

A subset $A$ of $X$ is \textit{$p$-limited} ($1 \leq p < \infty$) if for every weak* $p$-summable sequence $(x_n') \subset X'$, there is $a = (a_n) \in \ell_{p}$ such that $|x_n'(x)| < a_n$ for all $x \in A$ and $n \in \N$.
\smallskip

The aim of this note is to look at the concept of limited set  in a more natural way just by replacing $c_0$ by $\ell_p, 1 \leq p < \infty$. The following results motivate our study and the coming Definition \ref{plim4}.

\begin{prop} \label{plim2}
If $A$ is a $p$-limited set, $1 \leq p < \infty$, then $T(A)$ is relatively compact in $\ell_p$ for all $T \in L(X; \ell_p)$.
\end{prop}

\proof
Let $T: X \to \ell_p$ be the bounded operator defined by $T(x) = (x_n'(x))$. Since $(x_n')$ is weak* $p$-summable in $X'$ and $A$ is $p$-limited, there exists $(a_i) \in \ell_p$ such that $|x_n'(x)| \leq a_n$ for all $x \in A$ and $n \in \N$. Note that, for each $n$, $\sum_{i=n}^\infty |x_n'(x)|^p \leq \sum_{i=n}^\infty |a_i|^p$ for all $x \in A$, what implies that
$$s_n = \sup \conj{\sum_{i=n}^\infty |x_n'(x)|^p  }{x \in A} \leq \sum_{i=n}^\infty |a_i|^p \quad \text{for all $n \in \N$}.$$
As $(a_n) \in \ell_p$, $\sum_{i=n}^\infty |a_i|^p \to 0$ as $n \to \infty$. Therefore $s_n \to 0$, and consequently, $T(A)$ is relatively compact in $\ell_p$ \cite[Ex. 15, pg. 168]{alip}.
\qed

\smallskip

 The following is an example of a set that satisfies the thesis in Proposition \ref{plim2} which it is not $p$-limited.

\begin{ex} \label{plim3}
Let $1 \leq p < \infty$. By Pitt's theorem, every bounded operator $T: c_0 \to \ell_p$ is compact, i.e. $T(B_{c_0})$ is a compact set in $\ell_p$ for all $T \in L(c_0; \ell_p)$. However, $B_{c_0}$ cannot be a $p$-limited set, because it is not relatively weakly compact \cite[Proposition 2.1]{delgado}.
\end{ex}

\begin{df} \label{plim4}
Let $1 \leq p < \infty$. We say that a subset $A$ of $X$ is a \textit{coarse $p$-limited set} if~ $T(A) \subset \ell_p$ is a relatively compact set for all $T \in L(X; \ell_p)$.
\end{df}
It follows from Proposition \ref{plim2} that every $p$-limited set is coarse $p$-limited. However, in every infinite dimensional Banach space $X$ there are coarse $p$-limited sets that are not $p$-limited: Indeed, according to \cite[Corollary 3.3]{delgado}, there are compact sets in $X$ that are not $p$-limited. This remarkable difference led us to choose the word coarse in our definition.

In Section 2 we prove the basic results of the coarse $p$-limited sets. Their relationship with compact and weakly compact sets is discussed in Section 3. For instance, in the spaces $L_1(\mu)$ the coarse $1$-limited sets are relatively compact if and only if $L_1(\mu)$ is a Schur space. Also,
all weakly compact sets in $X$ are coarse $p$-limited sets if and only if every $T\in L(X;\ell_p)$ is completely continuous. Our notation and terminology is standard; we refer to \cite{habala} for background on Banach space theory.

\section{Basic results}

\begin{rema} \label{plim5}
\begin{enumerate}
    \item In Example \ref{plim3}, we saw that the unit ball $B_{c_0}$ is a coarse $p$-limited set in $c_0$ that is neither $p$-limited, $1 \leq p < \infty$ nor limited.

    \item For $1 \leq p < q < \infty$, the unit ball $B_{\ell_q}$ is a coarse $p$-limited set in $\ell_q$ that is not $p$-limited. In fact, by Pitt's theorem, every bounded operator $T \in L(\ell_q; \ell_p)$ is compact. Consequently, $B_{\ell_q}$ is coarse $p$-limited. However, since every $p$-limited set in $\ell_q$ is relatively p-compact \cite[Proposition 3.6]{delgado} and p-compactness implies compactness, we have that $B_{\ell_q}$ cannot be $p$-limited.

    \item The class of relatively compact sets in $\ell_p$ coincides with the class of coarse $p$-limited sets.

    \item If $A \subset X$ is limited, then $A$ is coarse $p$-limited. Indeed, if $T: X \to \ell_p$ is a bounded operator, then $T(A)$ is a limited set in $\ell_p$. Since $\ell_p$ has the Gelfand-Phillips property, it follows that $T(A)$ is relatively compact in $\ell_p$. Hence $A$ is a coarse $p$-limited set in $X$.

\end{enumerate}
\end{rema}

In particular, we have the following scheme:
$$ \text{limited or $p$-limited} \quad \Rightarrow \quad \text{coarse $p$-limited} \quad \nRightarrow \quad \text{limited nor $p$-limited}. $$

The following proposition concerning basic properties of coarse $p$-limited sets is immediate.
Compare with the corresponding results for  $p$-limited sets \cite[Proposition 2.2]{karnsinha2} and \cite[Proposition 2.1]{delgado}, among them that $p$-limited sets are relatively weakly compact.


\begin{prop}
 \label{plim6}
 Let $A, B$ be two subsets of a Banach space $X$.
 \begin{enumerate}
     \item If $B$ is coarse $p$-limited and $A \subset B$, then $A$ is coarse $p$-limited.

     \item If $A$ is coarse $p$-limited, then $\overline{A}$ is coarse $p$-limited.

     \item If $A$ and $B$ are coarse $p$-limited, then $A \cap B$, $A + B$ and $A \cup B$ are coarse $p$-limited.

     \item  If $A$ is a coarse $p$-limited and if $T \in L(X; Y)$, then $T(A)$ is a coarse $p$-limited set in $Y$.

    \item If $A$ is a coarse $p$-limited, then its circled convex hull, $aco \, (A),$ is a coarse $p$-limited sets as well.

 \end{enumerate}
\end{prop}

\begin{proof}
$1.$ If $T: X \to \ell_p$ is a bounded operator, then $T(B)$ is relatively compact in $\ell_p$. Since $A \subset B$, we have that $T(A) \subset T(B)$, what implies that $T(A)$ is a relatively compact set as well. Thus $A$ is coarse $p$-limited.

$2.$ If $T: X \to \ell_p$ is a bounded operator, then $\overline{T(A)}$ is compact in $\ell_p$. Since $T(\overline{A}) \subset \overline{T(A)}$, we get that $T(\overline{A})$ is relatively compact. Hence $\overline{A}$ is coarse $p$-limited.

$3.$ If $T: X \to \ell_p$ is a bounded operator, then $T(A)$ and $T(B)$ are relatively compact sets. Consequently, $T(A + B) = T(A) + T(B)$ and $T(A) \cup T(B)$ are relatively compact sets as well. Since $T(A \cap B) \subset T(A)$, $T(A + B) = T(A) + T(B)$ and $T(A\cup B) \subset T(A) \cup T(B)$, we have that $A \cap B$, $A + B$ and $A \cup B$ are coarse $p$-limited sets.

$4.$ If $S: Y \to \ell_p$ is a bounded operator, then $S \circ T: X \to \ell_p$ is a bounded operator. Since $A \subset X$ is coarse $p$-limited, $ST(A)$ is relatively compact in $\ell_p$. Therefore $T(A)$ is coarse $p$-limited.

$5.$ If $T: X \to \ell_p$ is a bounded operator, then $T(A)$ is a relatively compact set. By Mazur's theorem, the circled convex hull of $T(A)$ is a relatively compact set.
Since
 $T(aco \, (A)) \subset aco \, (T(A))$ it follows that $aco \, (A)$ is a coarse $p$-limited set.
\qed
\end{proof}

\smallskip

It follows from Remark \ref{plim5} that there exists coarse $p$-limited sets that are not relatively weakly compact. On the other hand, it also follows that $B_{\ell_2}$ is a coarse $1$-limited  set which is not coarse $2$-limited. Therefore, we do not have a version of \cite[Proposition 2.2]{delgado} for coarse $p$-limited sets.

\smallskip

Bourgain and Diestel showed that every limited set is conditionally weakly compact \cite{bourgdiest}. It is natural to wonder if every coarse $p$-limited set conditionally weakly compact. In the next proposition, we give a positive answer in the case $2 \leq p < \infty$.  Nevertheless, the result may fail for $p=1.$

\smallskip

\begin{prop} \label{plim7}
If~ $2 \leq p < \infty$ and if $A \subset X$ is a coarse $p$-limited set, then $A$ is conditionally weakly compact.
\end{prop}
\begin{proof}
Let $A \subset X$ be a bounded set which is not conditionally weakly compact. Then there exists $(x_n) \subset A$ without  weakly Cauchy subsequences. By Roshental's theorem, there exists a subsequence $(x_{n_k})$ which is equivalent to the canonical basis $(e_k)$ of $\ell_1$. Consider $Y = \overline{[x_{n_k}]}$ and let $S: Y \to \ell_1$ be the linear isomorphism such that $S(x_{n_k}) = e_k$ for all $k \in \N$. Let $j: \ell_1 \to \ell_p$ be the natural inclusion. Note that $j$ can be written by $j = i_2 \circ i_1$ where $i_1: \ell_1 \to \ell_2$ and $i_2 = \ell_2 \circ \ell_p$ are canonical inclusions. By Grothendieck's theorem \cite[pg. 15]{diestelbook}, $i_1$ is a 1-summing operator. Consequently, $j$ is a 1-summing operator as well by the ideal property \cite[pg. 37]{diestelbook}. It follows from the Inclusion theorem \cite[pg. 39]{diestelbook} that $j$ is a 2-summing operator. Again by the ideal property, $T = j \circ S: Y \to \ell_p$ is a 2-summing operator such that $j(S(x_{n_k})) = e_k$ for all $k$. Thus, we can extend $T$ to the Banach space $X$ \cite[Theorem 4.15]{diestelbook}. Finally, if $\conj{x_{n_k}}{k \in \N}$ was a coarse $p$-limited set, then $(e_k)$ would be a coarse $p$-limited set in $\ell_p$ by Proposition \ref{plim6}-4, what contradicts Remark \ref{plim5}-3.
\qed
\end{proof}

\medskip

Recall that every bounded operator $T: C(K) \to \ell_1$ is compact, see for instance \cite[II.D. Exercise 5]{W}. Therefore, the unit ball of $C(K)$ is a coarse 1-limited set. Pe{\l}czy\'{n}ski and Semadeni \cite[Main theorem]{PS} proved that the unit ball of $C(K)$ is conditionally weakly compact if and only if $K$ is dispersed, that is, it contains no non-void perfect set.
 So the unit ball of $C([0,1])$ (and many others) is a coarse $1$-limited set that is not conditionally weakly compact.
Thus, Proposition \ref{plim7}  does not hold for all $p$'s.

\smallskip

\smallskip

By Proposition \ref{plim7}, we get that the class of coarse $p$-limited sets with $2 \leq p < \infty$ is an intermediate class between the limited sets and the conditionally weakly compact sets. Recall that a Banach space $X$ has the DP* property if and only if every weakly compact set in $X$ is limited \cite{cargalou}, or equivalently, every conditionally weakly compact set is limited \cite[Proposition 1.2]{cargalou}. Hence if X has the DP* property, then the classes of limited sets, coarse $p$-limited sets, with $2 \leq p < \infty$, and conditionally weakly compact sets coincide in $X$.

\smallskip

\begin{rema} \label{plim8} In general, one cannot establish an inclusion relationship between  the class of
coarse $p$-limited sets and the class of coarse q-limited sets for $p\neq q.$
\end{rema}
\begin{proof} We noticed above that $B_{\ell_2}$, is a coarse 1-limited which is not coarse 2-limited.

Let $X=L_p$ with $2\leq p,$ so it has type $2,$ (see \cite{LT} page 73), and hence enjoys property $P_{1/2},$ according to \cite[Theorem 5.1]{alencar}. Then the unit ball of $L_p$ is a coarse $\frac{3}{2}$-limited set since every $T:L_p \to \ell_{3/2}$ is compact as it follows from \cite[Corollary 3]{alencar}. However, since $L_p$ contains $\ell_2$ as a complemented subspace (\cite[Proposition 6.4.2]{AK}, for instance), the corresponding projection maps $B_{L_{p}}$ onto a neighborhood of $0,$ preventing this image from being relatively compact. Hence $B_{L_{p}}$ is not a coarse $2$-limited set. \qed
\end{proof}

\begin{prop} \label{plim9}
Let $1 \leq p, q < \infty$. Then, the unit ball $B_{\ell_q}$ is a coarse $p$-limited set if and only if $p < q$.
\end{prop}
\begin{proof}
We proved in Remark \ref{plim5}-2 that if $p < q$, then  $B_{\ell_q}$ is a coarse $p$-limited set in $\ell_q$. Reciprocally, let $p \geq q$ and let $(e_k) \subset B_{\ell_q}$ be the canonical basis. If $T: \ell_q \to \ell_p$ is the inclusion map, then $Te_k = e_k$ for all $k$. Since $(e_k)$ is not relatively compact in $\ell_p$, it follows that $(e_k)$ is not a coarse $p$-limited in $\ell_q$. Hence $B_{\ell_q}$ cannot be coarse $p$-limited if $p \geq q$. \qed
\end{proof}

\smallskip

\begin{prop} \label{plim10}
The class of coarse $p$-limited sets satisfies the Grothendieck's encapsulating property, i.e. if for every $\epsilon>0$ there is a coarse $p$-limited set $A_\epsilon \subset X$ such that $A\subset A_\epsilon +\epsilon B_X,$ then $A$ is a coarse $p$-limited set.
\end{prop}

\begin{proof}
If $T: X \to \ell_p$ is a bounded operator, then $T(A) \subset T(A_{\epsilon}) + \epsilon T(B_X)$. Since $A_{\epsilon}$ is coarse $p$-limited, we have that $T(A_{\epsilon})$ is relatively compact. Consequently, $T(A)$ is relatively compact as well  what implies that $A$ is a coarse $p$-limited set. \qed
\end{proof}
\smallskip

We recall now the notion of \textit{polynomial} on a Banach space $X$ for the reader's convenience. 
A $k$-homogeneous polynomial $P: X\to Y,$ $Y$ a normed space, is the restriction to the diagonal in $X^k$ of some $k$-linear continuous mapping $A\in L(^k X;Y),$ that is, $P(x)=A(x, \stackrel{k}\dots, x),\, x\in X.$

\begin{prop} \label{plim11} Let $N\in\mathbb{N}$ and let $(x_n)\subset X$ be such that $\lim_n P(x_n)=P(x)$ for some $x\in X$ and all  $P\in \mathcal{P}(^{\leq N}X),$ the space of scalar polynomials of degree not greater than $N.$ Then $\{x_n: n\in \mathbb{N}\}$ is  a coarse $p$-limited set in $X$ for $p\leq N.$
\end{prop}
\begin{proof} 
The assumption means that for every polynomial $Q\in \mathcal{P}(^{\leq N}\ell_p)$ and any $T\in L(X,\ell_p),$  $\lim_n (Q\circ T)(x_n)=(Q\circ T)(x).$
Using a recurrence argument we deduce that $\lim_n Q\big(T(x_n)- T(x)\big)=0$ for all $Q\in \mathcal{P}(^{\leq N} \ell_p).$ Indeed, its immediate that $\lim_n Q\big(T(x_n)- T(x)\big)=0$ for all $Q\in \mathcal{P}(^{1} \ell_p) = \ell_p'$. If $Q \in P(^2 \ell_p)$ and $A$ is its associated 2-linear form,  we obtain that
\begin{align*}
    Q(Tx_n - Tx) & =A(Tx_n - Tx, Tx_n - Tx) \\
        & = A(Tx_n, Tx_n - Tx) - A(Tx, Tx_n - Tx) \\
        & = A(Tx_n, Tx_n) - A(Tx_n, Tx)- A(Tx, Tx_n - Tx).
\end{align*}
We have that
$\lim_n A(Tx_n, Tx_n) = \lim_n Q(Tx_n) = \lim_n Q(Tx) = \lim_n A(Tx, Tx)$, $\lim_n A(Tx_n, Tx) = \lim_n A_{Tx} (Tx) = \lim_n A_{Tx} (Tx) = \lim_n A(Tx, Tx)$ and
$\lim_n A(Tx, Tx_n - Tx) = \lim_n A_{Tx} (Tx_n - Tx) = 0$, what implies that
$$ \lim_n Q(Tx_n - Tx) = 0.$$
In a very similar way, one can show that
 assuming that $\lim_n Q\big(T(x_n)- T(x)\big)=0$ for all $Q\in \mathcal{P}(^{\leq N - 1} \ell_p)$, one can check that $\lim_n Q\big(T(x_n)- T(x)\big)=0$ for all $Q\in \mathcal{P}(^{N} \ell_p).$ We refrain from giving more details on this because the argument is the one already used.

Hence by \cite[6.3. Theorem]{CCG}, we conclude that $T(x_n) \to T(x).$ Since the set $\{T(x_n): n\in \mathbb{N}\}$ is relatively compact, we are done. \qed
\end{proof}
\medskip

The above result fails for $p>N,$ as shown by Pe{\l}czy\'{n}ski-Pitt theorem \cite[4.1]{alencar} because all $k$-linear forms on $\ell_p,$ with $k\leq N,$ are weakly sequentially continuous, thus they transform the unit basis into a null sequence. However, the unit basis is not coarse $p$-limited in $\ell_p.$

\smallskip
As a consequence of Proposition \ref{plim11}, if  on the Banach space $X$ all polynomials are weakly sequentially continuous, then every weakly convergent sequence in $X$ is a coarse $p$-limited sequence for all $1\leq p <\infty.$ Such Banach spaces $X$ were called $\mathcal{P}$-spaces  in \cite{CGG}. Among them, one finds those having the Dunford-Pettis property and others, remarkably without the  Dunford-Pettis property, like the dual of  Schreier's space or some Lorentz sequence space $d(w;1),$ see \textit{loc. cit.}

\smallskip

We conclude this section by showing that the coarse $p$-limited sets are not preserved by polynomials.

\begin{ex} \label{pol1}
Consider the polynomial $P: \ell_2 \to \ell_1$ given by $P((a_n)_n) = (a_n^2)_n$. By Remark \ref{plim5}, $B_{\ell_2}$ is coarse 1-limited in $\ell_2$. On the other hand, the set $P(B_{\ell_2})$ is not coarse 1-limited in $\ell_1$. Indeed, the unit basis $(e_n)$ is contained in $P(B_{\ell_2})$, since $P(e_n) = e_n.$ However it does not have any convergent subsequence in $\ell_1$.
\end{ex}

\section{Relation to compactness and weak compactness}

In this section, we seek to establish DP*-type and the Gelfand-Phillips type properties concerning the class of coarse $p$-limited sets. We begin with a new DP*-type property:

\begin{df} \label{pdp*1}
We say that a Banach space $X$ has the \textsl{coarse $p$-DP* property } if every relatively weakly compact set is coarse $p$-limited.
\end{df}

Since every limited set is coarse $p$-limited, every Banach space with the DP* property has the coarse $p$-DP* property  for all $1 \leq p < \infty$. In the course of the proof of Remark \ref{plim8}, we verified that $L_p$ does not have the coarse $2$-DP* property for $p \geq 2$.

\smallskip
We note that if $B_X$ is a coarse $p$-limited set, then $X$ has the coarse $p$-DP* property. However, this condition does not characterize the coarse $p$-DP* property. Indeed,
$\ell_1$ has the coarse $1$-DP* property, but $B_{\ell_1}$ is not a coarse 1-limited set.

\smallskip

Recall that reflexive infinite dimensional Banach spaces cannot have the DP* property. As a consequence of Proposition \ref{plim9}, we get that $\ell_q$ has the coarse $p$-DP* property  if and only if $1 \leq p < q < \infty$. Besides, as $B_{c_0}$ is a coarse $p$-limited set for all $1 \leq p < \infty$, we have that $c_0$ has the coarse $p$-DP* property  for all $1 \leq p < \infty$, although $c_0$ does not have the DP* property.

\begin{rema} \label{rema1} Banach spaces with property $\mathcal{P}$ have the coarse $p$-DP* property for all  $1 \leq p < \infty.$ In particular the Dunford-Pettis property implies the coarse $p$-DP* property for all  $1 \leq p < \infty.$ \end{rema}
  \begin{proof}
 Let $A\subset X$ be a weakly compact set. If $A$ is not  coarse $p$-limited, there is $T:X\to \ell_p$ and a sequence $(a_n)\subset A$ such that $(T(a_n))$ does not have convergent subsequence. Choose a weakly convergent subsequence  $(a_{n_k})\subset A.$ According to Proposition \ref{plim11},  $(a_{n_k})$ is coarse $p$-limited, hence  $T(a_{n_k})$ is a relatively compact sequence. A contradiction.\qed
  \end{proof}
\smallskip

As a consequence of Remark \ref{rema1}, we get that $L_1[0,1]$ has the coarse p-DP* property for all $1 \leq p < \infty$, even though $L_1[0,1]$ does not have the DP* property.

\smallskip

Our next Theorem resembles a similar characterization in \cite{cargalou}.

\begin{teo} \label{pdp*2}
A Banach space $X$ has the coarse $p$-DP* property
if and only if every bounded operator $T: X \to \ell_p$ is completely continuous.
\end{teo}

\begin{proof}
Assume that $X$ has the coarse $p$-DP* property  and let $T: X \to \ell_p$ be a bounded operator. Let $(x_n') \subset X'$ be the weak* $p$-summable sequence such that $T(x) = (x_n'(x))$. If $x_k \cvf 0$ in $X$, then $\{ x_k\}$ is a relatively weakly compact set, hence a coarse $p$-limited subset of $X$, what implies that $T(\{x_k\})$ is a relatively compact subset of $\ell_p$, i.e. $\sup \conj{\sum_{i=n}^\infty |x_i'(x_k)|^p}{k \in \N} \to 0$.
Given $\epsilon > 0$, there exists $n_0 \in \N$ such that $\sum_{i = n_0 + 1}^\infty |x_i'(x_k)|^p < \epsilon/2$ for all $k \in \N$. Since $x_k \cvf 0$, $|x_i'(x_k)| \to 0$ as $k \to \infty$ for all $i \in \N$. In particular, for each $i = 1, \dots, n_0$, there exists $k_i$ such that $|x_i'(x_k)| < \epsilon/(2n_0)$ for all $k \geq k_i$. If $k_0 = \max \{ k_1, \dots, k_{n_0} \}$ , we get that
\begin{align*}
    \norma{Tx_k}_p ^p & = \sum_{i=1}^\infty |x_i'(x_k)|^p = |x_1'(x_k)|^p + \cdots + |x_{n_0}'(x_k)|^p + \sum_{i = n_0 + 1}^\infty |x_i'(x_k)|^p
         < \epsilon,
\end{align*}
for all $k \geq k_0$. Hence $\norma{Tx_k}_p \to 0$ as $k \to \infty$.

Conversely, assume that every bounded operator from $X$ into $\ell_p$ is completely continuous. Let $A \subset X$ be a relatively weakly compact set and let $T: X \to \ell_p$ be a bounded operator. If $(x_n) \subset A$, then $(x_n)$ has a weakly convergent subsequence, what implies that $(Tx_n)$ has a convergent subsequence in $\ell_p$. Hence $T(A)$ is a relatively compact set. Thus $A$ is a coarse $p$-limited set.
\qed
\end{proof}

\medskip

It was proved in \cite{cargalou} that a Banach space $X$ has the DP* property if and only if every conditionally weakly compact set in $X$ is limited. We ask ourselves for a similar characterization for the coarse $p$-DP* property.

\begin{teo} \label{pdp*3} A Banach space $X$ has the coarse $p$-DP* property if, and only if,
   every conditionally weakly compact subset of $X$ is coarse p-limited.
       \end{teo}
\begin{proof}
Assume that $A \subset X$ is a conditionally weakly compact set that is not coarse $p$-limited. Then, there exists a bounded operator $T: X \to \ell_p$ such that $T(A)$ is not a relatively compact set in $\ell_p$. If $T(x) = (x_n'(x))$, we have that the decreasing sequence
$$ s_n = \sup \conj{\sum_{i=n}^\infty |x_i'(x)|^p}{x \in A} $$
is not null. So there is $\epsilon>0,$ such that $s_n > \epsilon^p$ for all $n$. This implies that there exists $(x_n) \subset A$ such that $\sum_{i=n}^\infty |x_i'(x_n)|^p \geq \epsilon^p$ for all $n$. Since $A$ is a conditionally weakly compact set, we can assume, without loss of generality, that $(x_n)$ is a weak Cauchy sequence.
As $\sum_{i=1}^\infty |x_i'(x_n)|^p < \infty$, there exists $i_n \geq n$ such that $\sum_{i \geq i_n} |x_i'(x_n)|^p < (\epsilon/2)^p$, thus by the second triangle inequality
\begin{align*}
    \|\big(x_i'(x_{i_n} - x_{n})\big)_{i=i_n}\|_p & \geq \|\big(x_i'(x_{i_n})\big)_{i=i_n}\|_p - \|\big(x_i'( x_{n})\big)_{i=i_n}\|_p \\
        & = \Big(\sum_{i = i_n}^\infty |x_i'(x_{i_n})|^p\Big)^{1/p}-\Big(\sum_{i = i_n}^\infty |x_i'(x_{n})|^p\Big)^{1/p} \\
        & > \epsilon/2
\end{align*}
On the other hand, since $\{x_{i_n} - x_n\}_n$ is a relatively weakly compact set, it must be coarse $p$-limited, what implies that
$$ \sup \conj{\sum_{i=k}^\infty |x_i'(x_{i_n} - x_n)|^p}{n \in \N} \to 0. $$
This yields
a contradiction.
\qed\end{proof}

\begin{cor} \label{pdp*4}
If $X$ does not contain a copy of $\ell_1$ and $X$ has the coarse p-DP* property, then $B_X$ is a coarse p-limited set.
\end{cor}

\begin{proof}
If $X$ does not contain copy of $\ell_1$, then $B_X$ is a conditionally weakly compact set, what implies that $B_X$ is a coarse p-limited set by Theorem \ref{pdp*3}. \qed
\end{proof}

\smallskip

Recall that a Banach space $X$ has the DP* property if and only if every polynomial $P \in \mathcal{P}(X; c_0)$ is completely continuous. We cannot generalize this fact to the coarse $p$-DP* property. Indeed, we know that $\ell_2$ has the coarse $1$-DP* property, even though the polynomial $P: \ell_2 \to \ell_1$ in Example \ref{pol1} is not completely continuous.

\smallskip

Now we study the relation between coarse $p$-limited sets and relatively compact sets.

\begin{df} \label{pgp1}
A Banach space $X$ is said to have the {\textsl{coarse p-Gelfand-Phillips property }} if every coarse $p$-limited subset of $X$ is relatively compact.
\end{df}

\begin{rema} \label{pgp2}
\begin{enumerate}
    \item We begin by showing that the Gelfand-Phillips property does not imply the coarse p-Gelfand-Phillips property. Indeed, we already know that the unit ball of $\ell_2$ is a coarse 1-limited set that is not relatively compact. Consequently, $\ell_2$ is a reflexive Banach space (hence it has the Gelfand-Phillips property) that does not have the coarse 1-Gelfand-Phillips property.

    \item It follows from Remark \ref{plim5}-3 that for $1\le p <\infty,$ $\ell_p(\Gamma),\;  \Gamma \text{ any index set, }$ has the coarse p-Gelfand-Phillips property. {\rm Recall that any element in $\ell_p(\Gamma)$ vanishes outside a countable subset of $\Gamma,$ so the same happens for a sequence. Therefore if $(a_n)$ is a coarse $p$-limited sequence in $\ell_p(\Gamma),$ it lies in a (complemented) copy of $\ell_p.$ As $(a_n)$ is a coarse $p$-limited sequence in $\ell_p,$ it is a relatively compact sequence in $\ell_p\subset \ell_p(\Gamma).$}

    \item Since limited sets are coarse $p$-limited, we get that if $X$ has the coarse p-Gelfand-Phillips property, for some $1 \leq p < \infty$, then $X$ has the Gelfand-Phillips property. In particular, $\ell_\infty$ does not have the coarse p-Gelfand-Phillips property  for any $1 \leq p < \infty$.

    \item By Example \ref{plim3}, the unit ball of $c_0$ is a coarse $p$-limited set for all $1 \leq p < \infty$. Consequently, $c_0$ does not have the coarse p-Gelfand-Phillips property.


\end{enumerate}
\end{rema}

It is known that a Banach space $X$ has the Gelfand-Phillips property if and only if every weakly null limited sequence in $X$ is norm null. Here the sufficiency follows because limited sets are conditionally weakly compact. In the next result, we present a similar characterization when $2 \leq p < \infty$.

\begin{teo} \label{pgp3}
If a Banach space $X$  has the coarse p-Gelfand-Phillips property, then every weakly null coarse $p$-limited sequence in $X$ is norm null. The converse holds if $2 \leq p < \infty$.
\end{teo}

\begin{proof}
Assume that $X$ has the coarse p-Gelfand-Phillips property  and let $(x_n) \subset X$ be a weakly null coarse $p$-limited sequence. In particular, $\{ x_n \}$ is a relatively compact set such that $x_n \cvf 0$, thus $\norma{x_n} \to 0$.

Now assume that $2 \leq p < \infty$ and let $A \subset X$ be a coarse $p$-limited set. Assume that $A$ is not compact. Then there is $\epsilon>0$ and a sequence $(x_n) \subset A$ such that $\|x_n-x_m\|>\epsilon$ for $n\neq m.$ By Proposition \ref{plim7}, $(x_n)$ has a weakly Cauchy subsequence, that will also be denoted by $(x_n)$. Now, by Proposition \ref{plim6}-3, the set $A - A$ is also coarse $p$-limited. Consequently, we get that the weakly null sequence $(x_n - x_{n+1})$ is coarse $p$-limited, what implies that $\norma{x_n - x_{n+1}} \to 0.$ This leads to a contradiction. \qed

\end{proof}

\medskip


\smallskip

\begin{cor} \label{pgp4}
If $X$ has both the coarse $p$-DP* property  and the coarse p-Gelfand-Phillips property, then $X$ has the Schur property. The converse holds when $2 \leq p < \infty$.
\end{cor}

\begin{proof}
If $x_n \cvf 0$ in $X$, then $\{ x_n \}$ is a coarse $p$-limited set (because $X$ has the coarse $p$-DP* property). Now, it follows from Theorem \ref{pgp3} that $\norma{x_n} \to 0$.

Since every Schur space has the DP* property and the DP* property implies in the coarse $p$-DP* property, it remains us to check that every Schur space has the coarse p-Gelfand-Phillips property, when $2 \leq p < \infty$, what is an immediate consequence of the converse part in Theorem \ref{pgp3}.
\qed
\end{proof}
\begin{rema} \label{L1}  The space $L_1(\mu),~\mu \text{ a measure,} $ is a coarse $1$-Gelfand-Phillips space if, and only if, it is a Schur space. \end{rema}
\begin{proof} As pointed out in \cite[5.1]{AK}, we may assume that $\mu$ is a probability measure without loss of generality.

 Since $L_1(\mu)$ has the Dunford-Pettis property, Corollary \ref{pgp4} yields the necessary condition.

Suppose now that $L_1(\mu)$ is a Schur space. Let $(f_n)$ be a coarse $1$-limited sequence. We claim that it is a relatively weakly compact set. If not, we apply \cite[Theorem 5.2.9]{AK} to the set $\{f_n\}$ and obtain a sequence, that we denote the same, equivalent to the canonical basis of $\ell_1$ that spans a complemented subspace. If $A\in L\big(L_1(\mu);\ell_1\big)$ denotes such projection, $\{A(f_n)\}=\{f_n\}$ should be a relatively compact set. This is not possible being $\{f_n\}$
the canonical basis, so the claim is proved. To conclude, observe that if a coarse $1$-limited set is not relatively weakly compact, there must be a sequence in it failing to be relatively weakly compact, in contradiction to the claim.\qed
\end{proof}

\begin{rema} \label{pgp5}
If $1 \leq p < \infty$, then the dual Tsirelson's space $T'$ has the coarse $p$-DP* property , but it fails to have the coarse p-Gelfand-Phillips property.
\end{rema}
\begin{proof}
It follows from \cite[Prop 1.5, Ex. 1.9]{castsan} that every bounded operator from $\ell_{q}$ into the Tsirelson's space $T$ is compact for all $1 \leq q < \infty$. Consequently, every bounded operator from the dual Tsirelson's space $T'$ into $\ell_p$ is compact for all $1 < p < \infty$. Consequently, $B_{T'}$ is a coarse $p$-limited set for all $1 < p < \infty$, that is not compact. Therefore, $T'$ has the coarse $p$-DP* property , but it fails to have the coarse p-Gelfand-Phillips property.

In order to achieve the result for $p = 1$, note that since $T'$ is reflexive and $\ell_1$ has the Schur property, we get that every bounded operator from $T'$ into $\ell_1$ is compact. \qed 
\end{proof}

\medskip

The following result is an adaptation of a result from \cite{emma1}.

\begin{prop} \label{pgp6}
Let $2 \leq p < \infty$.
If $X'$ has the coarse p-Gelfand-Phillips property and $Y$ has the Schur property, then $L(X;Y)$ has the coarse p-Gelfand-Phillips property.
\end{prop}

\begin{proof}
If $L(X;Y)$ does not have the coarse p-Gelfand-Phillips property, then by Theorem \ref{pgp3}, there exists a weakly null coarse p-limited sequence $(T_n) \subset L(X; Y)$ which is not norm null. Without loss of generality, we can assume that $\norma{T_n} \geq \epsilon$ for all $n$. In particular, for each $n$, there exists $x_n \in S_X$ such that $\norma{T_n(x_n)} \geq \epsilon$. For a given $y' \in Y'$, consider the bounded operator $\Phi_{y'}: L(X;Y) \to X'$ given by $\Phi_{y'}(T) = T'y'$. Since $(T_n)$ is a weakly null coarse p-limited sequence, $(\Phi_{y'}(T_n))$ is a weakly null coarse p-limited sequence in $X'$. As $X'$ has the coarse p-Gelfand-Phillips property, it follows from Theorem \ref{pgp3} that $\norma{\Phi_{y'}(T_n)} \to 0$. Consequently,
$$ |y'(T_n(x_n))| = |\Phi_{y'}(T_n)x_n| \leq \norma{\Phi_{y'}(T_n)} \to 0. $$
Since $y' \in Y'$ was taken arbitrarily, we get that $(T_n(x_n))$ is a weakly null sequence in $Y$ that has the Schur property. Thus, $\norma{T_n x_n} \to 0$, a contradiction. \qed
\end{proof}

\begin{prop} \label{pgp7}
\begin{enumerate}
    \item If $X$ has the coarse $p$-Gelfand-Phillips property, then
   every closed subspace of $X$ has the coarse $p$-Gelfand-Phillips property.
    \item The direct sum of two coarse $p$-Gelfand-Phillips spaces has the $p$-Gelfand-Phillips property.
\end{enumerate}
\end{prop}

\begin{proof}
$1.$ Let $Y$ be a closed subspace of $X$ and
let $A \subset Y$ be a coarse p-limited set. If $i: Y \to X$ denotes the inclusion map, then $i(A)$ is a coarse $p$-limited set in $X$. Hence $i(A)$ is relatively compact in $X$. Now, since $i$ is an inclusion from $Y$ into $X$, we have that $A$ is relatively compact in $Y$. Thus $Y$ has the coarse $p$-Gelfand-Phillips property.

$2. $ Assume that $X = Y \oplus Z$ where $Y$ and $Z$ are Banach spaces with the coarse $p$-Gelfand-Phillips property. Let $T: X \to Y$ and $S: X \to Z$ be the respective bounded projections.
If $A \subset X$ is a coarse $p$-limited set, then $T(A)$ and $S(A)$ are coarse $p$-limited sets in $Y$ and $Z$ respectively. Thus, their are both relatively compact in $Y$ and $Z.$ Since $A = T(A) + S(A)$, we get that $A$ is relatively compact in $X$. \qed
\end{proof}

\smallskip

\begin{rema}
As a consequence of Proposition \ref{pgp7}, we can point out the following:
\begin{enumerate}
    \item Since $L_q([0,1]), ~1 \leq q < \infty$, contains copy of $\ell_2$ \cite[Proposition 6.4.2]{AK} and $\ell_2$ does not have the coarse $1$-Gelfand-Phillips property, then $L_q([0,1])$ cannot have the coarse $1$-Gelfand-Phillips property.

    \item For $1 < q < \infty$, since $L_q([0,1])$ contains a complemented copy of $\ell_q$ \cite[Proposition 6.4.1]{AK}, we have that $L_q([0,1])$ does not have the coarse $p$-Gelfand-Phillips property for all $1 \leq p < q$.
\end{enumerate}
\end{rema}

\smallskip

%
%
%
%

\section{The coarse $p$-limited operators}

It is natural to consider the class of bounded operators related to the  coarse $p$-limited sets.

\begin{df} \label{op1}
We say that $T: X \to Y$ is a \textit{coarse $p$-limited operator} if $T(B_X)$ is a coarse $p$-limited set in $Y$.
\end{df}

\begin{exs} \label{op2}
\begin{enumerate}
    \item Since limited sets and $p$-limited sets are coarse $p$-limited sets, we have that limited operators and $p$-limited operators are coarse $p$-limited operators.

    \item The identity operator $I_{c_0}$ is a coarse $p$-limited operator which is neither limited nor $p$-limited.

    \item It follows from Proposition \ref{plim9} that the identity operator $I_{\ell_q}$ is a coarse $p$-limited operator  if and only if $p < q$.

    \item Since coarse $p$-limited sets are preserved by bounded operators, if $B_X$ is a coarse $p$-limited set, then every bounded operator from $X$ into any Banach space $Y$ is coarse $p$-limited.

    \item If $B_Y$ is a coarse $p$-limited set, then every bounded operator from any Banach space $X$ into $Y$ is coarse $p$-limited.

\end{enumerate}
\end{exs}

\begin{prop} \label{op3}
The set of all coarse $p$-limited operators between two Banach spaces $X$ and $Y$ is a closed ideal in $L(X; Y)$.
\end{prop}

\begin{proof}
If $(T_k)\subset L(X; Y)$ is a sequence of coarse $p$-limited operators that converges to $T \in L(X; Y)$, then given $\epsilon > 0$, there exists $k_0 \in \N$ such that $T(B_X)\subset T_{k_0}(B_X) + \epsilon B_F.$ Since $T_{k_0}(B_X)$ is a coarse $p$-limited set, $T(B_X)$ is also coarse $p$-limited set because of the encapsulating property.

Let $T: X \to Y$ be a coarse $p$-limited operator and let $S: Y \to Z$ be any bounded operator. Since $T(B_X)$ is a coarse $p$-limited set in $Y$, it follows from Proposition \ref{plim6}-4 that $S\circ T(B_X)$ is a coarse $p$-limited set in $Z$. Thus $S \circ T$ is a coarse $p$-limited operator.

Now let $T: X \to Y$ be a coarse $p$-limited operator and let $R: W \to X$ be any bounded operator. We want to prove that $T(R(B_W))$ is a coarse $p$-limited set in $Y$. Indeed, since
$R(B_{W}) \subset \norma{R} \, B_X$, we have that
$ T(R(B_{W})) \subset \norma{R} \, T(B_X). $ As $T(B_X)$ is a coarse $p$-limited set in $Y$, by considering the bounded operator $(y \in Y \mapsto \norma{R} \, y)$, we get by Proposition \ref{plim6}-4 that $\norma{R} \, T(B_X)$ is a coarse $p$-limited set as well. Consequently, $T(R(B_W))$ is a coarse $p$-limited set in $Y$ as we stated.
\qed
\end{proof}

\begin{teo} \label{op4}
For a Banach space $X$, the following are equivalent:
\begin{enumerate}
    \item $X$ has the coarse $p$-DP* property .

    \item Every weakly compact operator $T: Z \to X$ is coarse $p$-limited for any Banach space $Z$.

    \item Every weakly compact operator $T: \ell_1 \to X$ is coarse $p$-limited.
\end{enumerate}
\end{teo}

\begin{proof}
$1. \Rightarrow 2.$ Assume that $X$ has the coarse $p$-DP* property . If $T: Z \to X$ is a weakly compact operator, then $T(B_X)$ is a relatively weakly compact set in $X$, hence a coarse $p$-limited set. Thus, $T$ is a coarse $p$-limited operator.

$2. \Rightarrow 3.$ Obvious.

$3. \Rightarrow 1.$ Let $A$ be a relatively weakly compact subset of $X$ and let $S \in L(X; \ell_p)$. In order to prove that $S(A)$ is a relatively compact subset of $\ell_p$, take $(x_n) \subset A$. It suffices to prove that $(Sx_n)$ has a convergent subsequence in $\ell_p$. Indeed, since $A$ is relatively weakly compact, we can assume, without loss of generality, that $(x_n)$ is weakly convergent in $X$. Consider the bounded operator $T: \ell_{1} \to X$ given by
$T(a_j)_j = \sum_j a_j x_j$ for all $(a_j)_j \in \ell_1$.
Since $T(B_{\ell_1})$ is the circled convex hull of the relatively weakly compact set $\{ x_n \}$, we get by Krein-Smulian's theorem that $T$ is a weakly compact operator. By assumption, it follows that $T$ is a coarse $p$-limited operator, which means that $T(B_{\ell_1})$ is a coarse $p$-limited set in $X$. Since $Te_n = x_n$ for all $n$, we have that $(x_n)$ is a coarse $p$-limited sequence in $X$ and, consequently, $\{ Sx_n \}_n$ is a relatively compact set in $\ell_p$. Therefore, $(Sx_n)$ has a convergent subsequence as we claimed. \qed
\end{proof}

\smallskip

It is important to point out that we cannot establish a relationship between the class of coarse p-limited operators and the class of coarse q-limited operators for $p \neq q$. Following Remark \ref{plim8}, we have that the identity operator $I_{\ell_2}$ is a coarse 1-limited operator but it is not a coarse 2-limited operator. On the other hand, the unit operator $I_{L_2}$ is a coarse $\frac{3}{2}$-limited operator that is not coarse 2-limited.

\smallskip

The argument used in the proof of Theorem \ref{pdp*3} allows us to prove a similar result to \cite[Theorem 2.3]{cargalou}.

\begin{teo} \label{op5}
Let $X$ and $Y$ be Banach spaces. Assume first that $X$ has the coarse $p$-DP* property and let $T: X \to Y$ be a  non-coarse $p$-limited. Then $X$ contains a sequence $(x_n)$ which is equivalent to the canonical basis of $\ell_1$. Moreover, if $Y$ also has the coarse $p$-DP* property, then $(Tx_n)$ has a subsequence that is equivalent to the canonical basis of $\ell_1$.
\end{teo}

\begin{proof}
We begin by assuming that $X$ has the coarse $p$-DP* property. If $T: X \to Y$ is not a coarse $p$-limited operator, then there exists a weak* p-summable sequence $(y_n') \subset Y'$ such that the sequence $s_n = \sup \conj{\sum_{i \geq n} |y_n'(Tx)|^p}{x \in B_X}$ is not null. Without loss of generality we can assume that there exists $\epsilon > 0$ and a sequence $(x_n) \subset B_X$ such that $\sum_{i \geq n} |y_n'(Tx_n)|^p \geq \epsilon^p$ for all $n$. By  way of contradiction, we assume that $(x_n)$ has a weakly Cauchy subsequence that will be denoted by $(x_n)$. Since $\sum_{i=1}^\infty |y_i'(Tx_n)|^p < \infty$ for every $n$, there exists $i_n \in \N$ such that $\sum_{i \geq k_n} |y_i'(Tx_n)|^p < (\epsilon/2)^p$. Thus,
\begin{align*}
    \norma{(y_i'(Tx_{i_n} - Tx_n))_{i = i_n}}_p & \geq  \norma{(y_i'(Tx_{i_n}))_{i = i_n}}_p - \norma{(y_i'(Tx_n))_{i=i_n}}_p \\
            & = \left ( \sum_{i=i_n}^\infty |y_i'(Tx_{i_n})| \right )^{1/p} - \left ( \sum_{i=i_n}^\infty |y_i'(Tx_{n})| \right )^{1/p} \\
            & > \epsilon.
\end{align*}
However, since the sequence $(x_{i_n} - x_n)$ is weakly null, the coarse $p$-DP* property of $X$ yields that $\conj{x_{i_n} - x_n}{n \in \N}$ is a coarse $p$-limited set, what implies that
$$ \conj{\sum_{i = k} |y_i' \circ T (x_{i_n} - x_n)|^p}{n \in \N} \to 0 \quad \text{ as } k \to \infty, $$
a contradiction. Hence $(x_n)$ has no weakly Cauchy subsequence. Now, by Roshental's theorem, $(x_n)$ has a subsequence, that will also be denoted by $(x_n)$, that is equivalent to the canonical basis of $\ell_1$.

We can now assume that $Y$ also has the coarse $p$-DP* property. The same argument used in the first part of this proof shows that $(Tx_n)$ has no weakly Cauchy subsequence. Again by Roshental's theorem, we get that there exists a subsequence $(Tx_{n_k})$ which is equivalent to the canonical basis of $\ell_1$. \qed
\end{proof}

\smallskip
The following corollary is an immediate consequence of Theorem \ref{op5}.

\begin{cor} \label{op6}
If $X$ has the coarse $p$-DP* property and it does not contain copy of $\ell_1$, then every bounded operator from $X$ into any other Banach space $Y$ is coarse $p$-limited.
\end{cor}

\smallskip

The following result characterizes the coarse $p$-Gelfand-Phillips property in terms of coarse $p$-limited operators.

\begin{prop}
For a Banach space $X$, the following assertions are equivalent:
\begin{enumerate}
    \item $X$ has the coarse $p$-Gelfand-Phillips property.

    \item For every Banach space $Y$, every coarse $p$-limited operator from $Y$ into $X$ is compact.

    \item Every coarse $p$-limited operator from $\ell_1$ into $X$ is compact.
\end{enumerate}
\end{prop}

\begin{proof}
$1. \Rightarrow 2.$ and $2. \Rightarrow 3.$ are obvious.

$3. \Rightarrow 1.$ Let $A$ be a coarse $p$-limited susbet of $X$. In order to prove that $A$ is relatively compact, let $(x_n) \subset A$ and consider the bounded operator $T: \ell_1 \to X$ given by $T(a_j)_j = \sum_{j} a_j x_j$ for all $(a_j)_j \in \ell_1$. Since $T(B_{\ell_1}) \subset aco \, (T \{ x_n \}_n)$, we get from Remark \ref{plim6} that $T(B_{\ell_1})$ is a coarse $p$-limited set. Hence, by assumption, $T$ is a compact operator. Since $(x_n) \subset T(B_{\ell_1})$, we have that $(x_n)$ contains a convergent subsequence. Therefore $A$ is a relatively compact set, what implies that $X$ has the coarse $p$-Gelfand-Phillips property. \qed
\end{proof}

\smallskip

\noindent \textbf{Acknowledgments: } This paper is part of the second author Ph. D. thesis supervised by both Prof. Mary Lilian Louren\c co and Prof. Pablo Galindo whom he warmly thanks for their support and encouragement.

\end{document}